\documentclass[12pt]{amsart}
\usepackage{geometry}
\usepackage{graphicx}
\usepackage{amssymb}
\usepackage{epstopdf}
\usepackage{amsmath,amscd}
\usepackage{amsthm}
\usepackage{enumerate}
\usepackage{url,verbatim,soul}
\usepackage{subfig}
\usepackage{enumitem}

\RequirePackage[colorlinks,citecolor=blue,urlcolor=blue]{hyperref}
\usepackage{breakurl}
\theoremstyle{plain}
\newtheorem{theorem}{Theorem}
\newtheorem{definition}[theorem]{Definition}
\newtheorem{lemma}[theorem]{Lemma}
\newtheorem{proposition}[theorem]{Proposition}
\newtheorem{corollary}[theorem]{Corollary}

\newtheorem{remark}[theorem]{Remark}

\newtheorem{conjecture}[theorem]{Conjecture}

\newcommand\es{\varnothing}

\newcommand\ol{\overline}
\newcommand\Aut{\mathrm{Aut}}

\newcommand\sQ{{\mathcal Q}}
\newcommand\sH{{\mathcal H}}

\newcommand\TT{{\mathbb T}}

\newcommand\sF{{\mathcal F}}
\newcommand\RR{{\mathbb R}}
\newcommand\ZZ{{\mathbb Z}}

\newcommand\PP{{\mathbb P}}

\newcommand\om{\omega}

\newcommand\be{\beta}

\newcommand\De{\Delta}
\newcommand\qq{\qquad}
\newcommand\q{\quad}
\newcommand\resp{respectively}

\newcommand\oo{\infty}
\newcommand\sG{{\mathcal G}}

\newcommand\sT{{\mathcal T}}

\newcommand\Ga{\Gamma}
\newcommand\Om{\Omega}
\newcommand\La{\Lambda}

\newcommand\lra{\leftrightarrow}

\newcommand\pc{p_{\text{\rm c}}}
\newcommand\pcs{\pc^{\text{\rm site}}}
\newcommand\pu{p_{\text{\rm u}}}
\newcommand\pus{\pu^{\text{\rm site}}}
\newcommand\pcb{\pc^{\text{\rm bond}}}
\newcommand\pub{\pu^{\text{\rm bond}}}

\renewcommand\ell{l}

\newcommand\pd{\partial}

\newcommand\sm{\setminus}

\newcounter{mycount}\newcounter{mycount2}
\newenvironment{romlist}{\begin{list}{\rm(\roman{mycount2})}%
   {\usecounter{mycount2}\labelwidth=1cm\itemsep 0pt}}{\end{list}}

\newenvironment{letlist}{\begin{list}{\rm(\alph{mycount})}%
   {\usecounter{mycount}\labelwidth=1cm\itemsep 0pt}}{\end{list}}

\newcommand\wtilde{\widetilde}
\newcommand\ga{\gamma}

\newcommand\what{\widehat}

\newcommand\nst{non-self-touching}

\numberwithin{equation}{section}
\numberwithin{theorem}{section}
\numberwithin{figure}{section}

\title{Hyperbolic site percolation}

\author{Geoffrey R.\ Grimmett}
\address{(GRG) Centre for
Mathematical Sciences, Cambridge University, Wilberforce Road,
Cambridge CB3 0WB, UK} 
\email{grg@statslab.cam.ac.uk}
\urladdr{\url{http://www.statslab.cam.ac.uk/~grg/}}
\address{(ZL) Department of Mathematics,
University of Connecticut, Storrs, Connecticut 06269-3009, USA} \email{zhongyang.li@uconn.edu}
\urladdr{\url{http://www.math.uconn.edu/~zhongyang/}}

\author{Zhongyang Li}
\date{2 March 2022, revised 23 June 2024}

\begin{document}

\begin{abstract}
Several results are presented for site percolation on quasi-transitive, planar graphs $G$ with one end, when properly
embedded in either the Euclidean or hyperbolic plane. If 
$(G_1,G_2)$ is a matching pair derived from some quasi-transitive mosaic $M$, 
then $\pu(G_1)+\pc(G_2)=1$,
where $\pc$ is the critical probability for the existence of an infinite cluster, 
and $\pu$ is the critical value for the existence of a \emph{unique} such cluster.
This fulfils and extends to the hyperbolic plane an observation of Sykes and Essam (1964),
and it extends to quasi-transitive site models a theorem of Benjamini and Schramm 
(Thm 3.8, \emph{J.\ Amer.\ Math.\ Soc.} 14 (2001) 487--507) for transitive bond percolation.

It follows that $\pu (G)+\pc (G_*)=\pu(G_*)+\pc(G)=1$, where 
$G_*$ denotes the matching graph of $G$. 
In particular,  $\pu(G)+\pc(G)\ge 1$ and hence, when $G$ is amenable we have $\pc(G)=\pu(G) \ge \frac12$.
When combined with the main result of the companion paper by the same authors (\lq Percolation
critical probabilities of matching lattice-pairs', \emph{Random Struct.\ Alg.} (2024)),
we obtain for transitive $G$ that the strict inequality $\pu(G)+\pc(G)> 1$ holds
if and only if $G$ is not a triangulation.

A key technique is a method for expressing a planar site percolation process on a matching pair
in terms of a dependent bond process on the corresponding dual pair of graphs.
Amongst other matters, the results reported here answer positively two conjectures of Benjamini and Schramm 
(Conj.\ 7, 8, \emph{Electron.\ Commun.\ Probab.} 1 (1996) 71--82)
in the case of quasi-transitive graphs.

\end{abstract}
\maketitle

\section{Introduction and results}

\subsection{Percolation on planar graphs}\label{sec:int}

Percolation was introduced in 1957 by Broadbent and Hammersley \cite{BH57} as a model for the spread of fluid through a random medium. Percolation provides a natural mathematical setting for such topics as the study of disordered materials, magnetization, and the spread of disease. See \cite{HDC18,GP,LP} for recent accounts of the theory. 
We consider here site percolation on a graph $G=(V,E)$, assumed to be infinite, locally finite, connected, and planar. 
 The current work has two linked objectives.  

Our major objective is to study the relationship between the percolation critical point $\pc$ 
and the critical point $\pu$ marking the existence of a \emph{unique} infinite cluster. More specifically, 
we establish the formula $\pus(G_1)+\pcs(G_2)=1$ for a matching pair $(G_1,G_2)$  of graphs arising from a
quasi-transitive mosaic, appropriately embedded in either the Euclidean or hyperbolic plane.
See Section \ref{ssec:mainres1}.

Setting $(G_1,G_2)=(G,G_*)$ above, with $G_*$ the matching graph
of $G$, we obtain 
$$
\pus(G)+\pcs(G_*)=\pus(G_*)+\pcs(G)=1.
$$
It follows that $\pus(G)+\pcs(G)>1$ if and only if the strict inequality $\pcs(G_*)<\pcs(G)$ holds.
Necessary and sufficient conditions for the last inequality are established in \cite{GL-match}
(for transitive graphs) and \cite{G24} (for quasi-transitive graphs).
For transitive $G$, this implies that $\pus(G)+\pcs(G)>1$ if and only if $G$ is not a triangulation.

Our second objective, which is achieved in the process of proving the above formula, is to validate 
Conjectures 7 and 8 of Benjamini and Schramm \cite{bs96} concerning the existence of infinitely many infinite clusters.
Details of these conjectures are found in Section \ref{sec:mainresBS}.

The organization of the paper is presented in Section \ref{ssec:summary}.

\subsection{Critical points of matching pairs}\label{ssec:mainres1}

Since loops and multiple edges have no effect on the existence of infinite clusters in site percolation,  the graphs considered in this article are generally assumed to be \emph{simple}
(whereas their dual graphs may be non-simple).
The main results proved in this paper are as follows (see Sections \ref{sec:def}--\ref{ssec:percnot} 
for explanations of the standard
notation used here).  

The word \lq transitive' shall mean \lq vertex-transitive' throughout this work.
We denote by
\begin{align*}
\sG: \q&\text{all infinite, locally finite, planar, $2$-connected, simple graphs},\\
\sT:  \q&\text{the subset of $\sG$ containing all such transitive graphs},\\
\sQ: \q&\text{the subset of $\sG$ containing all such quasi-transitive graphs.}
\end{align*}
Since the work reported here concerns matching and dual graphs, 
the graphs in $\sG$ will be considered in their plane embeddings. 
The most interesting such graphs turn out to be those with one end.
We shall recall in Section \ref{sec:emb} that one-ended graphs in $\sT$ have unique proper embeddings in
the Euclidean/hyperbolic plane up to homeomorphism,  and hence their matching and dual graphs are uniquely defined.
The situation is more complicated for one-ended graphs in $\sQ$, in which case we fix a plane embedding of
$G\in\sQ$  for which the dual graph $G^+$ is quasi-transitive. Such an embedding is
called \emph{canonical}; if $G$ has connectivity $2$,
a canonical embedding need not be unique (even up to homeomorphism), 
but its existence is guaranteed by Theorem \ref{p21}(c).

\begin{figure}
\centerline{\includegraphics[width=0.85\textwidth]{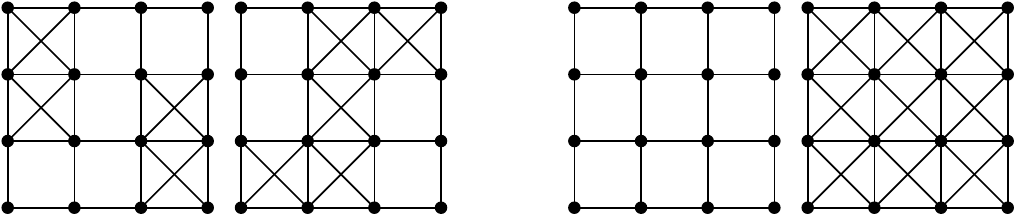}}
\caption{Two matching pairs derived from the square lattice $\ZZ^2$. 
Each $3\times3$ grid is repeated periodically about $\ZZ^2$. The pair on the right
generates $\ZZ^2$ and its matching graph.}\label{fig:mg}
\end{figure}

Matching pairs of graphs were introduced by Sykes and Essam \cite{SE64}
and explored further by Kesten \cite{K82}. Let $M\in\sQ$
be  one-ended and canonically embedded in the plane (we call $M$ a \emph{mosaic} following the earlier literature). 
Let $\sF_4=\sF_4(M)$ be the set of faces of $M$ bounded by $n$-cycles with $n\ge 4$,
and let $\sF_4=F_1\cup F_2$ be an arbitrary
 quasi-transitive partition of $\sF_4$. The graph $G_i$ is obtained from $M$ by adding 
all diagonals to all faces in $F_i$. The pair $(G_1,G_2)$ is
called a \emph{matching pair}.  
The \emph{matching graph} $G_*$ of a one-ended graph $G\in\sQ$ is
obtained by adding all diagonals to all faces in $\sF_4(G)$.  Thus, $(G,G_*)$ is an instance
of a  matching pair. Two examples of matching pairs are given in Figure \ref{fig:mg}.

The notation $\pu$ denotes the critical value for the existence of a \emph{unique} infinite cluster. 
Further notation and background for percolation is deferred to Section \ref{ssec:percnot}.

\begin{theorem}\label{m1}
\mbox{\hfil}
\begin{letlist}
\item Let $(G_1,G_2)$ be a matching pair derived from the mosaic $M\in\sQ$. We have that
\begin{equation}\label{eq:-5}
\pus (G_1)+\pcs (G_2) =1.
\end{equation}
\item 
Let $G\in\sQ$ be one-ended. Then
\begin{equation}\label{eq:-6}
\pus (G)+\pcs (G)\geq 1.
\end{equation} 
If $G$ is transitive, equality holds in \eqref{eq:-6} if and only if $G$ is a triangulation.
\end{letlist}
\end{theorem}

In the context of \eqref{eq:-5}, Sykes and Essam \cite[eqn (7.3)]{SE64} presented motivation for the exact formula 
\begin{equation}\label{eq:exact}
\pcs(G_1)+\pcs(G_2)=1,
\end{equation}
and this has been verified in a number of cases when $G$ is amenable (see \cite{vdB81}).
This formula does not hold for non-amenable graphs. Equation \eqref{eq:-5}
appears without proof in \cite[eqn (4)]{Mertens} for a restricted class of graphs.

\begin{remark}[Strict inequality]\label{rem:enh}
Equation \eqref{eq:-6}  follows from \eqref{eq:-5} with $(G_1,G_2)=(G,G_*)$,
by the inequality $\pcs(G)\ge \pcs(G_*)$. 
This weak inequality holds trivially since $G$ is a subgraph of $G_*$. The corresponding strict inequality 
$\pcs(G)>\pcs(G_*)$ is investigated in the companion papers \cite{G24,GL-match}, where
necessary and 
sufficient conditions are presented. 
By \eqref{eq:-5},
$$
\pus(G)-\pus(G_*) = \pcs(G)-\pcs(G_*) \ge 0,
$$
so that strict inequality for $\pcs$ is equivalent to strict inequality for $\pus$.
\end{remark}

\begin{remark}[Canonical embeddings]\label{rem:uniq}
When $G$ has connectivity $2$,
it may possess more than one canonical embedding; by Theorem \ref{m1},
$\pcs(G_*)$ and $\pus(G_*)$ are independent of the choice of canonical embedding.
This may be seen directly by observing that, in situations where there is a choice of embedding, 
the two-point connectivity functions are equal.
\end{remark}

\begin{remark}[Amenability]\label{rem:amen}
If $G\in\sQ$ is one-ended and in addition amenable,
by the uniqueness of the infinite cluster \cite{AKN,BK89},
we have $\pcs (G)=\pus (G)$; in this case,  $\pcs (G)\geq \frac{1}{2}$ 
by \eqref{eq:-6}. 
If $G$ is transitive, 
we have $\pcs(G)=\frac12$ if and only if $G$ is the usual amenable, triangular lattice. 
\end{remark}

The dual graph of a plane graph $G$ is denoted $G^+$.

\begin{remark}[Bond percolation]\label{rem:m3}
Theorem \ref{m1} may be compared with the corresponding results for bond percolation.
It is proved in \cite[Thm 3.8]{bs00} that 
$$
\pcb (G)+\pub (G^+)=1
$$ 
for any non-amenable, \emph{transitive} $G\in\sT$.
If, instead, $G\in\sT$ is amenable, it is standard that $\pub(G^+)=\pcb(G^+)=
1-\pcb(G)$.
These facts are extended to \emph{quasi-transitive} graphs in \cite[Thm 8.31]{LP}.
We make use of some elements of \cite{bs00, LP} here, while studying 
the more general site percolation directly
via the concept of pairs of matching graphs.
\end{remark}

\subsection{Existence of infinitely many infinite clusters}\label{sec:mainresBS}

A number of problems for percolation on non-amenable graphs were formulated by Benjamini and Schramm 
in their influential paper \cite{bs96}, including the following two conjectures. 

\begin{conjecture} [\mbox{\cite[Conj.\ 7]{bs96}}]\label{c12} 
Consider site percolation on an infinite, connected,  planar graph $G$ with minimal degree at least $7$.  Then, for any $p\in(\pcs, 1-\pcs)$,  we have $\PP_p(N=\infty)=1$.   Moreover,  it is the case that $\pcs<\frac 12$, so the above interval is invariably non-empty.
\end{conjecture}

It was proved in \cite[Thm 2]{HP19} that $\pcs<\frac12$  for planar graphs with vertex-degrees at least $7$. 

\begin{conjecture} [\mbox{\cite[Conj.\ 8]{bs96}}]\label{c11} 
Consider site percolation on a planar graph $G$ satisfying $\PP_{\frac12}(N\ge 1)=1$. Then $\PP_{\frac12}(N=\infty)=1$.
\end{conjecture}

Percolation in the hyperbolic plane was later studied by Benjamini and Schramm \cite{bs00}. In the current paper, we extend certain results of \cite{bs00} to amenable planar graphs and to site percolation, and we confirm Conjectures \ref{c12} and \ref{c11} for all planar, quasi-transitive graphs.

Conjectures \ref{c12} and \ref{c11} were verified  in \cite{ZL17} when $G$ is a regular triangular tiling (or \lq triangulation') 
of the hyperbolic plane $\sH$ for which each vertex has degree 
at least $7$. A significant property of a triangulation  is that its matching graph is the same as the original graph.

The next two theorems establish Conjectures \ref{c12} and \ref{c11} for planar, quasi-transitive graphs.

\begin{theorem}\label{sc1}
Consider site percolation on a graph $G\in\sQ$, each vertex of which has degree $7$ or more. 
\begin{letlist}
\item For every $p\in (\pcs , 1-\pcs )$,  there  exist, $\PP_p$-a.s., infinitely  many  
infinite  $1$-clusters and infinitely many infinite $0$-clusters.
\item For every $p\in [0,1]$, there exists, $\PP_p$-a.s., at least one infinite cluster 
that is either a $1$-cluster or a $0$-cluster.
\end{letlist}
\end{theorem}

\begin{theorem}\label{sc2} 
Consider site percolation on $G\in\sQ$, and assume that $\PP_{\frac12}(N\ge 1)=1$. 
Then, $\PP_{\frac12}$-a.s., there exist infinitely many infinite $1$-clusters and infinitely many infinite $0$-clusters.
\end{theorem}

The approach to establishing Conjectures \ref{c12} and \ref{c11} is to classify $\sQ$ according to
amenability and the number of ends, and then prove these conjectures for each such subclass of graphs. 
We recall the following well-known theorem.

\begin{theorem}[\mbox{\cite{Ho}, \cite[Prop.\ 2.1]{Bab97}}]\label{p13}
A graph $G$ that is  infinite, connected, locally finite, and quasi-transitive has either one or two 
or infinitely many ends. If it has two ends, then it is amenable. If it has infinitely many ends, 
then it is non-amenable.
\end{theorem}

Let $G\in\sQ$. By Theorem \ref{p13},
only the following cases may occur.
\begin{romlist}
\item $G$ is amenable and one-ended. This case includes the square lattice, for which percolation has been studied extensively; see, for example, \cite{GP,K82}.
\item $G$ is non-amenable and one-ended. It is proved in \cite{bs00} that $\pcs <\pus $ and $\pcb <\pub $ for this case.
\item $G$ has two ends, in which case there is no percolation phase transition of interest.
\item $G$ has infinitely many ends.
\end{romlist}
We shall study percolation on each class of graphs listed above. 
Matching graphs and dual graphs  will play important roles in our analysis.

\subsection{Organization of material}\label{ssec:summary}

 Section \ref{sec:notation} is devoted to basic notation for graphs and percolation. 
In Section \ref{sec:bk}, we review certain known results that will be used to prove the main results of 
Section \ref{ssec:mainres1}.   It is explained in Section \ref{sec:sbond} how a site percolation process on
a planar graph may be expressed in terms of a dependent bond process on the same graph; this allows a connection
between site percolation on the matching graph and bond percolation on the dual graph. 
We prove Theorem \ref{m1}(a) for amenable graphs in Section \ref{s7},
and for non-amenable graphs in Section \ref{ss5}.  
Theorem \ref{sc1} is proved in Section \ref{psc1}, and 
Theorem \ref{sc2} in Section \ref{psc2}. 

\section{Notation}\label{sec:notation}

\subsection{Graphical notation}\label{sec:def}
Let $\Aut(G)$ be the automorphism group of the graph $G=(V,E)$. A graph $G$ is called \emph{vertex-transitive}, or simply \emph{transitive}, if all the vertices lie in the same orbit under the action of $\Aut(G)$. The graph $G$ is called \emph{quasi-transitive} if the action of $\Aut (G)$ on $V$ has only finitely many orbits. It is called \emph{locally finite} if all vertex-degrees are finite.
An edge with endpoints $u$, $v$ is denoted $\langle u,v\rangle$, in which case
we call $u$ and $v$ \emph{adjacent} and
we write $u\sim v$.
The graph-distance $d_G(u,v)$ between vertices $u$, $v$ is the minimal number of edges in
a path from $u$ to $v$. 

A graph $G$ is \emph{planar} if it can be embedded in the plane $\RR^2$  in such a way that its edges intersect only at their endpoints; a planar embedding of such $G$ is called a \emph{plane} graph.  
A \emph{face} of a plane graph $G$ is an (arc-)connected 
component of the complement $\RR^2\sm G$. Note that faces are 
open sets, and may be either bounded or unbounded. 
With a face $F$, we associate the set of vertices and edges in its boundary.
The \emph{size} of a face is the number of edges in its boundary.
While it may be helpful to
think of a face as being bounded by a cycle of $G$, the reality can be more complicated
in that faces are not invariably simply connected (if $G$ is disconnected) and their boundaries 
are not generally self-avoiding cycles or paths (if $G$ is not $2$-connected).
A plane graph $G$ is called a \emph{triangulation} it every face is bounded by a $3$-cycle.

A manifold $M$ is called \emph{plane} if it is a surface and, 
for every self-avoiding cycle $\pi$ of $M$,
$M\sm\pi$ has exactly two connected components. When a graph is drawn in a plane manifold $M$, the terms embedding and face mean the same as when embedded in the Euclidean plane.
We say that an embedded graph $G\subset M$  is \emph{properly embedded} if every compact subset of $M$ contains only finitely many vertices of $G$ and intersects only finitely many edges. 
Henceforth, all embeddings will be assumed to be proper.
The term \emph{plane}
shall mean either the Euclidean plane or the hyperbolic plane, and each may be denoted $\sH$ when appropriate.

A \emph{cycle} (or \emph{$n$-cycle}) 
$C$ of a simple graph $G=(V,E)$ is a sequence $v_0,v_1,\dots,  v_{n+1}=v_0$ of vertices $v_i$ such that $n \ge 3$, 
$e_i:=\langle v_i,v_{i+1}\rangle$ satisfies
$e_i\in E$ for $i=0,1,\dots,n$, and $v_0,v_1,\dots,v_n$ are distinct.
Let $G$ be a plane graph, properly embedded in $\sH$.
In this case we write $C^\circ$ for the bounded component of $\RR^2\sm C$,
and $\ol C$ for the closure of $C^\circ$.
The \lq matching graph' $G_*$ is obtained from $G$ by adding all possible diagonals to every face of $G$. That is, let $F$ be such a face, and let $\pd F$ be the set of vertices lying in 
the boundary of $F$. We augment $G$ by adding edges between any distinct pair 
$x, y\in V$ such that (i)  there exists a face $F$ such that $x,y \in\pd F$ and
(ii) $\langle x,y\rangle \notin E$.  We write $D$ for the set of diagonals, so that
$G_*=(V,E\cup D)$.
We recall from \cite[Thm 3]{Kr} (see Remark \ref{rem:emb}(d)) that, for a $2$-connected graph $G$,
every face is bounded by either a cycle or a doubly-infinite path.

Next we define a matching pair. Let $M\in\sQ$
be one-ended (we follow the earlier literature by calling $M$ a \emph{mosaic} in this context).
By the forthcoming Remark \ref{rem:emb}(d), $M$ has an
embedding in the plane such that the dual graph $M^+$ and the matching graph $M_*$ are quasi-transitive,
and furthermore every face of $M$ is bounded by a cycle. Let $\sF_4=\sF_4(M)$ be the set of faces of $M$
bounded by $n$-cycles with $n\ge 4$,
and let $\sF_4=F_1\cup F_2$ be a partition of $\sF_4$. The graph $G_i$ is obtained from $M$ by adding 
all diagonals to all faces in $F_i$, and we assume that $\Aut(M)$ has some subgroup $\Ga$ that acts
quasi-transitively on each $G_i$. The pair $(G_1,G_2)$ is
said to be a \emph{matching pair} derived from $M$. 

The graph $G$ is called \emph{amenable} if its Cheeger constant satisfies
\begin{equation}
\inf_{K\subseteq V,\, |K|<\infty}\frac{|\De K|}{|K|}=0,\label{am}
\end{equation}
where $\De K$ is the subset of $E$ containing edges with exactly one endpoint in $K$. If the left side of \eqref{am} is strictly positive, the graph $G$ is called \emph{non-amenable}.

Each $G\in\sT$  is quasi-isometric with one and only one of the following spaces: $\ZZ$, the 3-regular tree, the Euclidean plane, and the hyperbolic plane; see \cite{Bab97,bs00}.
 See \cite{CFKP,Iver} for background on hyperbolic geometry.

Recall that the number of ends of a connected graph is the supremum over its finite subgraphs $F$ of the number of
infinite components that remain after removing $F$, and recall Theorem \ref{p13}. 
The number of ends of a graph is highly relevant to properties of statistical mechanical models on the graph; see \cite{GrL,ZLSAW}, for example, for discussions of the relevance of the number of ends to the number and speed of self-avoiding walks.

\subsection{Percolation notation}\label{ssec:percnot}
 Let $G=(V,E)$ be a connected, simple graph with bounded vertex-degrees. 
 A \emph{site percolation} configuration on $G$ is an assignment $\omega\in \Om_V:=\{0,1\}^{V}$ to each vertex of either state 0 or state 1. A cluster in $\omega$ is a maximal connected set of vertices in which each vertex has the same state. A cluster may be a 0-cluster or a 1-cluster depending on the common state of its vertices, and it may be finite or infinite. We say that \lq percolation (or $1$-percolation) occurs' in $\omega$ 
 if there exists an infinite $1$-cluster in $\omega$. For $\om\in\Om_V$, we write $1-\om$ for the
configuration with open/closed inverted.

A \emph{bond percolation} configuration $\om\in \Om_E:=\{0,1\}^{E}$ is an assignment to each edge in $G$ of either state 0 or state 1.   A bond percolation model may be considered as a site percolation model on the so-called \emph{covering graph} (or \emph{line graph}) $\wtilde G$ of $G$. Therefore, we may use the term 1-cluster (\resp, 0-cluster) for a maximal connected set of edges with state $1$ (\resp, state $0$) in a bond configuration. The \emph{size} of a cluster in site/bond percolation is the number of its vertices.

We call a vertex or an edge \emph{open} if it has state 1, and \emph{closed} otherwise. Let $\mu$ be a probability measure on $\Om_V$ endowed with the product $\sigma$-field. The corresponding site model is the probability space  $(\Om_V,\mu)$, with a similar definition for a bond model $(\Om_E,\mu)$. 
The central questions in percolation theory concern the existence and multiplicity of infinite clusters viewed as functions of $\mu$. 

A percolation model $(\Om,\mu)$ is called \emph{invariant} if $\mu$ is invariant under the action of $\Aut(G)$. 
An invariant measure is called \emph{ergodic} if there exists an automorphism subgroup $\Ga$ acting quasi-transitively on $G$ such that 
$\mu(A)\in\{0,1\}$ for any $\Ga$-invariant event $A$. See, for example, \cite[Prop.\ 7.3]{LP}. It is standard that the product measure $\PP_p$ is ergodic if $G$ is infinite and quasi-transitive.

Consider percolation on a graph $G=(V,E)$.
A site or bond configuration $\om$ induces open and closed subgraphs of $G$ 
 in the usual way, and we write $N$ ($=N_G(\om)$)
 for the number of infinite $1$-clusters, 
 and $\ol N$  ($=\ol N_G(\om)$) for the number of infinite $0$-clusters. For site percolation on a graph $G$, 
 we write  $N_*$, $\ol N_*$ for the corresponding quantities on the matching graph $G_*$. 
 A configuration is in one--one
correspondence with the set of elements (vertices or edges, as appropriate) that are open in the configuration.

Let $p\in [0,1]$. We endow $\Om_V$ with the product measure $\PP_p$ with density $p$.
For $v\in V$, let $\theta_v(p)$ be the probability that $v$ lies in an infinite open cluster.
It is standard that there exists $\pcs(G)\in(0,1]$ such that
$$
\text{for } v\in V, \qq \theta_v(p) \begin{cases} = 0 &\text{if } p<\pcs(G),\\
>0 &\text{if } p > \pcs(G),
\end{cases}
$$ 
and $\pcs(G)$ is called the \emph{(site) critical probability} of $G$.

More generally, consider (either bond or site) percolation on a graph $G$ with probability measure $\PP_p$.
The corresponding critical points may be expressed as follows:
\begin{align*}
\pcs (G)&:=\inf\{p\in[0,1]: \PP_p(N\ge 1)=1 \text{ for site percolation} \},\\
\pcb (G)&:=\inf\{p\in[0,1]: \PP_p(N\ge 1)=1 \text{ for bond percolation}\},
\end{align*}
and
\begin{align*}
\pus (G)&:=\inf\{p\in[0,1]: \PP_p(N=1)=1 \text{ for site percolation} \},\\
\pub (G)&:=\inf\{p\in[0,1]: \PP_p(N=1)=1 \text{ for bond percolation}\}.
\end{align*}
By the Kolmogorov zero--one law, $\PP_p(N\ge 1)$ equals either $0$ or $1$.

The notation $\pc$ (\resp, $\pu$) shall always mean the critical probability $\pcs$ (\resp, $\pus$) of the site model.
For background and notation concerning percolation theory, the reader is referred to the book \cite{GP}.

\section{Background}\label{sec:bk}
We review certain known results that will be used in the proofs of our main results.

\subsection{Embeddings of one-ended planar graphs}\label{sec:emb}

We say that the $2$-sphere, the Euclidean plane,  and the hyperbolic plane  constitute the \emph{natural geometries} (see,
for example, Babai \cite[Sect.\ 3.1]{Bab97}).
 The natural geometries are two-dimensional Riemannian manifolds. 
 An \emph{Archimedean tiling} of a two-dimensional Riemannian
manifold is a tiling by regular polygons such
that the group of isometries of the tiling acts transitively
on the vertices of the tiling.
An infinite, one-ended, transitive planar
graph can be characterized as a tiling of either the Euclidean plane or the hyperbolic plane. 

An \emph{embedding} of a graph $G=(V,E)$ (with underlying $1$-complex denoted $\vert G\vert$)
in a surface $M$ is a continuous map $\phi: \vert G\vert \to M$ such that the induced map $\vert G\vert
\to \phi(\vert G\vert)$ is a homeomorphism. An embedding $\phi$ is called \emph{cellular} if 
$M\setminus \phi(G)$ is a disjoint union of spaces 
homeomorphic to an open disc. (See \cite{Moh} and \cite[Sect.\ 3.2]{MT}.)

We shall consider embeddings of planar graphs in either the Euclidean or hyperbolic plane, and we use the notation
$\sH$ to denote either of these, as appropriate for the context.

\begin{theorem}\label{p21}\mbox{\hfil}
\begin{letlist}
\item \cite[Thms 3.1,  4.2]{Bab97}
If $G\in\sT$ is one-ended, then $G$ may be embedded in $\sH$
as an Archimedean tiling, and all automorphisms of $G$ extend to isometries of $\sH$.
If $G\in\sQ$ is one-ended and  $3$-connected, then $G$ may be embedded in $\sH$
such that all automorphisms of $G$ extend to isometries of $\sH$. 

\item \cite[p.\ 42]{Moh} Let $G$ be a $3$-connected graph, cellularly embedded in $\sH$ 
such that all faces are of finite size.
Then $G$ is uniquely embeddable in the sense that for any two cellular embeddings $\phi_1: G\to S_1$,
$\phi_2:G\to S_2$ into planar surfaces $S_1$, $S_2$, there is a homeomorphism $\tau:S_1\to S_2$ such that
$\phi_2=\tau\phi_1$.  

\item \cite[Thm 8.25 and proof, pp.\ 288, 298]{LP} 
If $G=(V,E)\in\sQ$ is one-ended, there exists some embedding of $G$ in $\sH$
such that the edges coincide with geodesics, the dual 
graph $G^+$ is quasi-transitive, and all automorphisms of $G$ extend to isometries of $\sH$. 
Such an embedding is called \emph{canonical}.

\item \cite{ST97} The automorphism group $\Aut(G)$ of a quasi-transitive graph $G$ with quadratic growth
contains a subgroup isomorphic to $\ZZ^2$ that acts quasi-transitively on $G$.
\end{letlist}
\end{theorem}

\begin{remark}\label{rem:emb}
Some known facts concerning embeddings follow.
\begin{letlist}
\item \cite[Props 2.2, 2.2]{BIW} All one-ended, transitive, planar graphs are $3$-connected,
and all proper embeddings of a one-ended, quasi-transitive, planar graph have only finite faces.
\item 
By Theorem \ref{p21}(b), any one-ended $G\in\sT$  
has a unique proper cellular embedding in $\sH$ up to homeomorphism.
Hence, the matching and dual graphs of $G$ are independent of the embedding.
\item  
The conclusion of part (b) holds for any one-ended, $3$-connected $G\in \sQ$. 
\item For a one-ended  $G\in\sQ$,  we fix
a canonical embedding (in the sense of Theorem \ref{p21}(c)). With this  given, 
the dual graph $G^+$ and the matching graph $G_*$ are quasi-transitive, 
and furthermore (by \cite[Thm 3]{Kr})  the boundary 
of every face is a cycle of $G$.
\end{letlist}
\end{remark}

\begin{remark}[Proper embedding]\label{rem:arch}
Theorem \ref{p21}(a) implies in particular that every one-ended $G\in\sT$ 
may be properly embedded in its natural geometry. Such an embedding
is called \emph{topologically locally finite} (TLF) by Renault \cite[Prop.\ 5.1]{DN}, \cite{R04}. 
For a related discussion in the case of non-amenable graphs, see \cite[Prop.\ 2.1]{bs00}.
\end{remark}

\begin{remark}[Connectivity]\label{rem:2con}
Graphs with connectivity $1$ have been excluded from membership of $\sG$ 
(and therefore from $\sT$ and $\sQ$ also).
Percolation on such graphs has little interest since any finite dangling ends may be removed 
without changing the existence of an infinite cluster. 
Moreover, let $F$ be a face of a mosaic $M$, such that $F$ contains some dangling end $D$. 
If $(G_1,G_2)$ is a matching pair derived from $M$,  the critical values $\pc(G_i)$ are unchanged
if $D$ is deleted.
\end{remark}

The representation of transitive, planar
graphs as tilings of natural geometries enables the development of universal techniques to study statistical mechanical models on all such graphs; see, for example, the study \cite{GrL} of a universal lower bound for connective constants on infinite, connected, transitive, planar, cubic graphs.

\subsection{Percolation}
We assume throughout this subsection that the graph $G$ is infinite, connected, and locally finite.

\begin{lemma}[\mbox{\cite[Cor.\ 1.2]{SR99}, \cite{HP}}]\label{lsu}
Let $G$ be quasi-transitive, and consider either site or bond percolation on $G$. Let $0<p_1<p_2\leq 1$, and assume that $\PP_{p_1}(N=1)=1$. Then $\PP_{p_2}(N=1)=1$.
\end{lemma}

\begin{definition}
Let $G=(V,E)$ be a graph. Given  $\om\in \Om_V$ and a vertex $v\in V$, write $\Pi_{v}\om=\om\cup\{v\}$ (which is to say that $v$ is declared open). For $A\subseteq \Om_V$, we write $\Pi_v A=\{\Pi_v \om: \om\in A\}$. A site percolation process $(\Om_V,\mu)$ on $G$ is called \emph{insertion-tolerant} if $\mu(\Pi_vA)>0$ for every $v\in V$ and every event $A\subseteq \Om_V$ satisfying $\mu(A)>0$.

 A site percolation is called \emph{deletion-tolerant} if $\mu(\Pi_{\neg v}A)>0$ whenever $v\in V$ and $\mu(A)>0$, where $\Pi_{\neg v}\om=\om\setminus\{v\}$ for $\om\in \Om_V$, and $\Pi_{\neg v}A= \{\Pi_{\neg v} \om: \om\in A\}$.
\end{definition}

Similar definitions apply to bond percolation.
We shall encounter weaker definitions in Section \ref{sec:weaker}.

\begin{lemma}[\mbox{\cite[Thm 7.8]{LP}, \cite[Thm 8.1]{BLPS2}}]\label{l32}
Let $G=(V,E)$ be a connected, locally finite,  quasi-transitive  graph, and let
$(\Om,\mu)$ be an invariant (site or bond) percolation on $G$. 
Assume either or both of the following two conditions hold:
\begin{letlist}
    \item $(\Om,\mu)$ is insertion-tolerant,
    \item $G$ is a non-amenable planar graph with one end.
\end{letlist}
Then $\mu(N\in\{0,1,\oo\})=1$.
If $\mu$ is ergodic, $N$ is $\mu$-a.s.\ constant.
\end{lemma}

The sufficiency of (a) is proved in \cite[Thm 7.8]{LP} for transitive graphs, and the same proof is valid
for quasi-transitive graphs.
The sufficiency of (b) is proved in \cite[Thm 8.1]{BLPS2}.

\subsection{Planar duality}\label{sec:weaker}

Let $G=(V,E)$ be a plane graph, and write $\sF$ for the set of its faces.  The dual graph $G^{+}=(V^+,E^+)$ is defined as follows. The sets  $V^+$ and $\sF$ are in one--one correspondence, written $v_f \lra f$. Two vertices $v_f,v_g\in V^+$ are joined by $n_{f,g}$ parallel edges where $n_{f,g}$ is the number of edges of $E$ 
common to the faces $f,g\in\sF$. Thus, $E^+$ and $E$ are in one--one correspondence, written $e^+ \lra e$.

For a bond configuration $\om\in \Om_E$, we define the dual configuration $\om^+\in\Om_{E^+}$ by: 
for each dual pair $(e,e^+)\in E\times E^+$ of edges, we have
\begin{equation}
\om(e)+\om^+(e^+)=1.\label{pdl}
\end{equation}
In the following, $(\Om_E,\mu)$ is a bond percolation model on $G=(V,E)$.
Similar definitions apply to site percolation.

\begin{definition}\label{df23}
A probability measure $\mu$ is called \emph{weakly insertion-tolerant} if there exists a function
$f: E\times \Om_E \rightarrow \Om_E$ such that
\begin{letlist}
\item for all $e$ and all $\om\in\Om_E$, we have $\omega\cup\{e\}\subseteq f(e,\omega)$,
\item for all $e$ and all $\om$, the difference $f(e,\omega)\setminus[\omega\cup\{e\}]$ is finite, and
\item for all $e$ and each event $A$ satisfying $\mu(A)>0$, the image of $A$ under $f(e,\cdot)$ is an
event of strictly positive probability.
\end{letlist}
\end{definition}

\begin{definition}\label{df24}
A probability measure $\mu$ is called \emph{weakly deletion-tolerant} if there exists a function
$h: E\times \Om_E \rightarrow \Om_E$ such that
\begin{letlist}
\item for all $e$ and all $\omega\in\Om_E$, we have $\omega\setminus\{e\}\supseteq h(e,\omega)$,
\item for all $e$ and all $\omega$, the difference $[\omega\setminus\{e\}]\setminus h(e,\omega)$ is finite, and
\item for all $e$ and each event $A$  satisfying $\mu(A)>0$, the image of $A$ under $h(e,\cdot)$ is an event of strictly positive probability.
\end{letlist}
\end{definition}

\begin{lemma}[\mbox{\cite[Thm\ 8.30]{LP}}] \label{ll25}
Let $G=(V,E)\in\sQ$ be non-amenable and one-ended, and consider $G$ embedded canonically in the plane 
(such an embedding exists by Theorem \ref{p21}(c)). Let $(\Om_E,\mu)$ be an invariant, ergodic, bond percolation 
on $G$, assumed to be both weakly insertion-tolerant and weakly deletion-tolerant. 
For $\om\in\Om_E$, let $N(\om)$ be the number of 
infinite open components of $\om$, and $N^+(\om)$ the number of infinite open components of the dual process
$\om^+$ given in \eqref{pdl}. 
Then 
\begin{equation*}
\mu\bigl((N, N^+)\in\{(0,1),(1,0),(\infty,\infty)\}\bigr) = 1.
\end{equation*}
\end{lemma}

\subsection{Graphs with two or more ends}\label{s6}

We summarise here the main results for critical percolation probabilities on multiply-ended graphs. 

\begin{theorem}[\cite{HPS, SR01}] \label{m2}
Let $G\in\sQ$ have two ends.  The critical percolation probabilities satisfy 
\begin{equation*}
\pcb (G)=\pcs (G)=\pub (G)=\pus (G)=1.
\end{equation*}
\end{theorem}

\begin{theorem}\label{m3}
Let $G\in\sQ$ have infinitely many ends. Then 
$$
\pcb(G)\le \pcs(G)<\pub (G)=\pus (G)=1.
$$
\end{theorem}

The standard inequality $\pcb\le\pcs$ holds for all graphs, and was stated in \cite{jmh61}. 
The corresponding strict inequality
was explored in \cite[Thm 2]{GS98} for bridgeless, quasi-transitive graphs. 
The equalities $\pub = \pus=1$ were proved for transitive graphs in \cite[eqn (3.7)]{SR01} (see also \cite{HPS}),
and feature in \cite[Exer.\ 7.9]{LP} for quasi-transitive graphs.
The inequality $\pcs< 1$ for non-amenable graphs was given in 
\cite[Thm 2]{bs96}.

\subsection{FKG inequality}

For completeness, we state the well-known FKG inequality. See, for example, \cite[Sect.\ 2.2]{GP} for further details.

\begin{theorem}[FKG inequality,  \cite{FKG, HR}]
Let $\mu$ be a strictly positive probability measure on $\Om_V$ satisfying the  \emph{FKG lattice condition}:
\begin{equation}
\mu(\omega_1\vee\omega_2)\mu(\omega_1\wedge\omega_2)\geq \mu(\omega_1)\mu(\omega_2),\qquad \omega_1,\omega_2\in \{0,1\}^{V}.\label{fkgl}
\end{equation}
For any increasing events $A,B\subseteq\{0,1\}^{V}$, we have that $\mu(A\cap B)\geq \mu(A)\mu(B)$.
\end{theorem}

\section{Planar site percolation as a bond model}\label{sec:sbond}

Let $M=(V,E)\in\sQ$  be a mosaic,
and let $(G_1,G_2)$ be a matching pair derived from $M$ according to the partition $\sF_4(M)=F_1\cup F_2$.
If $F_i\ne\es$, then $G_i$ is non-planar. This is an impediment to consideration of the dual graph of $G_i$, which
in turn is overcome by the introduction of so-called facial sites.

Let $\sF=\sF(M)$ be the set of faces of $M$ (following \cite{K82}, we include triangular faces).
The triangular faces of $\sF$ do not appear in $F_1\cup F_2 = \sF_4$, but we allocate each such face arbitrarily to
either $F_1$ of $F_2$ (for concreteness, we add them all to $F_1$). 
One may
replace the mosaic $M$ 
by the triangulation $\what M$ obtained by placing a \emph{facial site} $\phi(F)$ 
inside each face $F\in\sF$, 
and joining  $\phi(F)$ to each vertex in the boundary of $F$. 
(See \cite[Sec.\ 2.3]{K82} and \cite[Sect.\ 4.2]{GL-match}.)

When considering site percolation on $M$ (\resp, $M_*$), 
one declares the facial
sites of $\what M$ to be invariably closed (\resp, open). 
Site percolation on $G_i$ is equivalent to site percolation on $\what M$ subject to:
\begin{equation}\label{eq:-10}
\text{a facial site $\phi(F)$ is declared open if $F\in F_i$ and closed if $F\in\sF\sm F_i$}.
\end{equation}
Note that, if $F$ is a triangular face,  the state of $\phi(F)$  is independent of the connectivity of 
its other vertices.

The \emph{facial graph} $\what G_i$ is 
obtained by adding to $M$ the facial sites of $F_i$ only, together with their incident edges. 
We write $\what G_i=(V_i,E_i) := (V\cup \Phi_i,E\cup\eta_i)$ where $\Phi_i$ is the set of facial sites of $G_i$
and $\eta_i$ is the set of edges incident to facial sites.
We shall consider two site percolation processes, namely, 
percolation of open sites on $\what G_1$ and of closed sites on $\what G_2$ (subject to \eqref{eq:-10}).
To this end, for $\om\in\Om_V$,  let $\om_1$ (\resp, $\om_2$) be the site configuration on 
$\what G_1$ (\resp, $\what G_2$)  that agrees with $\om$ on $V$ and is
\emph{open} on $\Phi_1$ (\resp, \emph{closed} on $\Phi_2$).

Given $\om\in\Om_V$, we construct a bond configuration $\be_{\om_1}\in \Om_{ E\cup\eta_1}$ by 
\begin{equation}\label{eq:sitebond}
\be_{\om_1}(e)=\begin{cases} 1 &\text{if } \om_1(u)=\om_1(v) = 1,\\
0 &\text{otherwise},
\end{cases}
\end{equation}
where $e=\langle u,v\rangle\in E\cup\eta_1$. 
Let $\be_{\omega_1}^+ := 1- \be_{\om_1}$ be the corresponding dual configuration on the dual 
graph $\what G_1^+=(V_1^+, E_1^+)$ of $\what G_1$ as in \eqref{pdl}. 
Let $\what G_1(\be_{\om_1})$
(\resp, $\what G_1^+(\be_{\om_1}^+)$)
be the graph with vertex-set $V_1$ (\resp, $V_1^+$) endowed with the open edges of $\be_{\om_1}$
(\resp, $\be_{\om_1}^+$).
Note that, if $\omega$ has law $\PP_p$, then the law of $\be_{\om_1}$
is one-dependent. We may identify the vector $\be_{\om_1}$ with the set of its open edges.

\begin{lemma}\label{lem9}
Suppose $\omega\in\Om_V$ has law $\PP_p$ where $p\in(0,1)$.
The law $\mu$ of $\be_{\om_1}$ is weakly deletion-tolerant 
and weakly insertion-tolerant.
Moreover, $\mu$ is ergodic.
\end{lemma}

\begin{proof}
Let $e=\langle u,v\rangle\in E\cup\eta_1$  and $\om\in\Om_V$. For $w\in V$,
let $D_w$ be the set of edges of $\what G_1$ of the form $\langle w,x\rangle$ with $\om(x)=1$.
Select an endvertex, $u$ say, of $e$ that is not a facial site (such a vertex always exists), and define
$$
f(e,\be_{\om_1}) = \be_{\om_1}\cup (D_u\cup D_v\cup \{e\}),\qq 
h(e,\be_{\om_1}) = \be_{\om_1}\sm (D_u\cup\{e\}).
$$
The edge-configuration $f(e,\be_{\om_1})$ (\resp, $h(e,\be_{\om_1})$) is that obtained by setting 
$u$ and $v$ to be open (\resp, $u$ to be closed). With these functions $f$, $h$, the conditions of Definitions \ref{df23} and \ref{df24}
hold since $G$ is locally finite.
The ergodicity holds by the assumed quasi-transitivity of $G_1$ and the fact that $\PP_p$ is a product measure (see the comment in Section \ref{ssec:percnot}). 
\end{proof}

\begin{figure}
\centerline{\includegraphics[width=0.6\textwidth]{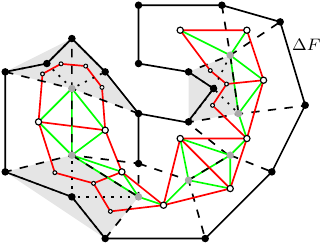}}
\caption{
An illustration of the one--one correspondence between $C_1(F)$ and $C_2(F)$ of Proposition \ref{prop:3}.
The black line is the boundary of the face $F$; the dashed lines are edges of $M$ inside $F$; 
the dotted lines are edges of $\eta_1$.  
The shaded regions are faces of $M$ that belong to $F_1$; the black points are open vertices; the grey points are closed vertices; the white points are dual vertices of $\what G_1$. 
The green graph is the $0$-cluster $C_2(F)$ of $\what G_2(\om)$ (i.e., the $1$-cluster of $\what G_2(1-\om)$)
that corresponds 
to the red cluster $C_1(F)$ of $\what G_1^+(\be^+_{\om_1})$.}\label{fig:color}
\end{figure}

For $\om\in\Om_V$,
let $\what G_1(\om)$ be the subgraph of $\what G_1$ induced by the set of $\om_1$-open vertices (that is, the set of $v$ with $\om_1(v)=1$), and define $\what G_2(\ol{\om})$ similarly in terms of closed vertices of $\om_2$ in $\what G_2$.

We make some notes concerning the relationship between $\what G_1(\om)$, $\what G_2(\ol\om)$, 
and $\what G_1^+(\be^+_{\om_1})$, as illustrated
in Figure \ref{fig:color}.  A \emph{cutset} of a graph $H$ is a subset of edges whose removal disconnects some 
previously connected component of $H$, and which
is minimal with this property. Recall that a \emph{face} of a plane graph $H=(V',E')$ is a connected 
component of
$\sH\sm E'$. A face $F$ can be bounded or unbounded, and it need not be simply connected.
It has a boundary $\De F$ comprising edges of $H$; 
even when $F$ is bounded and simply connected, the set $\De F$ of edges
need not be cycle of $H$ unless $H$ is $2$-connected.

\begin{proposition}\label{prop:3} 
Let $M=(V,E)\in\sQ$ be one-ended and embedded canonically in $\sH$.
Let $\om\in\Om_V$,
and let $F$ be a face (either bounded or unbounded) of $\what G_1(\om)$. 
\begin{letlist}
\item Let $C$ be a cycle (\resp, doubly-infinite path) of $\what G_1(\om)$. 
The set of  edges of $\what G_1^+$ intersecting
$C$ forms a finite (\resp, infinite) cutset of $\what G_1^+$. 

\item The set $F\cap V_1^+$ of dual vertices of $\what G_1$ inside $F$, together with the set of
open edges of $\be_{\om_1}^+$ lying inside $F$, forms a non-empty,
connected component $C_1(F)$ of $\what G_1^+(\be^+_{\om_1})$. 

\item The set $F\cap(V\cup \Phi_2)$ of vertices of $\what G_2$ inside $F$ forms a 
(possibly empty) $0$-cluster $C_2(F)$ of $\what G_2(\ol\om)$. 

\item 
Either each of $F$, $C_1(F)$, $C_2(F)$ is bounded or each is unbounded.
\end{letlist}
\end{proposition}

\begin{proof}
(a) This is immediate by planar duality.

(b) Note first that every vertex $w$ of $M$ inside $F$ satisfies $\om(w)=0$.
Since $F$ is bounded by a cycle of $\what G_1$, it is a non-empty, disjoint union $F=\bigcup_{i\in I} A_i$
of faces $A_i$ of $\what G_1$ (more precisely, the two sides of the equality differ on a set of 
Lebesgue measure $0$). Each $A_i$ is bounded, and 
contains a (unique) dual vertex $d_i$. It is
standard that the dual set $D=\{d_i: i\in I\}$ induces a connected 
graph $C_1(F)$ in $F$. 
Since
no edge $f$ of $C_1(F)$ intersects $\De F$, 
we have $\be_{\om_1}^+(f)=1$ for all such $f$.

(c)  It can be the case that $F\cap (V\cup\Phi_2)=\es$, in which case we take 
$C_2(F)$ to be the empty graph (this is the situation if and only if $F$ is a triangular
face of $\what G_1(\om)$). 
Suppose henceforth that
$F\cap (V\cup\Phi_2)\ne\es$ and note as above that $\om(w)=0$ for every  $w\in F\cap(V\cup \Phi_2)$.
It is a standard property of matching pairs of graphs that $F\cap (V\cup\Phi_2)$ induces a connected 
subgraph $C_2(F)$ of $F\cap \what G_2$.

Parts (b) and (c) make use of two so-called \lq standard' properties, full discussions of
which are omitted here. It suffices to prove the \lq 
standard' property of \emph{matching} pairs, 
since the corresponding property for dual pairs then
follows by passing to covering (or line)  graphs (see, for example, \cite[Sec.\ 2.6]{K82}). For matching pairs, an early 
reference is \cite[App.]{SE64}, and a more detailed account is found in \cite[Sec.\ 3, App.]{K82}
(see, in particular, Proposition A.1 of \cite{K82}).
The latter assumes slightly more than here on the 
mosaic $M$, but the methods apply notwithstanding.

(d) When $F$ is finite, so must be $C_1(F)$ and $C_2(F)$, since the embedding 
of $M$ is proper.
When $F$ is infinite,  the same holds of $C_1(F)$ and $C_2(F)$, since the faces
of $G$ are uniformly bounded.
\end{proof}

Recall the notation $N_G(\om)$ from Section \ref{ssec:percnot}.

\begin{proposition}\label{ll2}
Let $M =(V,E)\in\sQ$ be one-ended and embedded canonically in $\sH$,
and let $\omega\in\Om_V$. Then,
\begin{equation}
\begin{gathered}\label{eq:-11}
N_{G_1}(\om) =N_{\what G_1}(\om) = N_{\what G_1}(\be_{\om_1}), \\
N_{G_2}(1-\om) =N_{\what G_2}(1-\om) = N_{\what G_1^+}(\be^{+}_{\om_1}).
\end{gathered}
\end{equation}
\end{proposition}

\begin{proof}
Equation \eqref{eq:-11} holds by the  definition of $\be_{\om_1}$, and from Proposition \ref{prop:3}
on noting (for given $\om$) the one--one correspondence between infinite clusters 
of $\what G_2(1-\om)$ and of $\what G_1^+(\be_{\om_1}^+)$.  The
facial site in any face of $M$ is a surrogate for the diagonals of that face.
\end{proof}

\begin{remark}[Conformality]\label{rem4}
It is classical that every bond percolation model may be phrased as a site model on the so-called 
covering (or line) graph
(see, for example, \cite[p.\ 24]{GP}). While the converse is generally false, using Definition \ref{eq:sitebond}
we obtain a one-dependent bond model from the site model on the same graph; furthermore, the connectivity
relations of these two processes are identical.
It was proved by Smirnov \cite{Smir} that critical site percolation on the triangular lattice $\TT$ satisfies Cardy's formula, 
and moreover has properties of conformal invariance (see also \cite{CN1,CN2}).
By the above observation, the dependent bond process on $\TT$ has similar properties, 
and its dual process on the hexagonal lattice.
\end{remark}

\section{Amenable planar graphs with one end}\label{s7}

In this section, we prove Theorem \ref{m1}(a) for amenable, one-ended graphs;
see Remark \ref{rem:enh} for an explanation of part (b) of the theorem.
It is standard that such graphs are properly embeddable in the Euclidean plane,
denoted $\sH$ in this section.

Recall first that, for any infinite, quasi-transitive, amenable graph $G$, 
and invariant, insertion-tolerant  measure $\mu$, the number $N_G$ of infinite open clusters 
satisfies $\mu(N_G\le 1)=1$ (see \cite[Thm 7.9]{LP} for the transitive case, the quasi-transitive case is similar).
As in Section \ref{ssec:percnot}, $\ol N_G$ denotes the number
of infinite closed clusters.

\begin{lemma}\label{l73}
Let $M=(V,E)\in\sQ$ be amenable, one-ended, and embedded canonically in $\sH$, 
and let $(G_1,G_2)$
be a matching pair derived from $M$. 
Let $(\Om_V,\mu)$ be an ergodic, insertion-tolerant site percolation on $M$ satisfying the FKG lattice condition \eqref{fkgl}. 
Then
\begin{equation}\label{eq:-9}
 \mu\bigl((N_M, \ol N_M) = (1,1)\bigr)=\mu\bigl((N_{G_1},\ol N_{G_2}) = (1,1)\bigr)=0.
\end{equation}
\end{lemma}

A pair $\gamma$, $\gamma'$ of isometries of $\RR^2$
is said to act in a 
\emph{doubly periodic manner} on $G$ (in its canonical embedding) if they generate a subgroup of
$\Aut(G)$ that is isomorphic to $\ZZ^2$, and the embedding is called \emph{doubly periodic}
if such a pair exists.  In preparation for the proof of Lemma \ref{l73}, we note the following.

\begin{theorem}\label{thm:embedqt}
Let $G\in\sQ$ be amenable and one-ended. A
canonical embedding of  $G$ in $\RR^2$ is doubly periodic.
\end{theorem}

\begin{proof}
This may be proved in a number of ways, including using either Bieberbach's theorem on crystalline groups
\cite{LB12,AV97}
or Selberg's lemma \cite{Alp}. 
Instead, we use a more direct  route via the main theorem of Seifter and Trofimov \cite{ST97}
(see Theorem \ref{p21}(d)).

Viewed as a graph, $G$ has quadratic growth. This standard fact holds as follows.
By \cite[Thm 1.1]{Bab97}, $G$ has either linear, or quadratic, or exponential growth. As noted at
\cite[Thm 9.3(b)]{Cam04}, being one-ended, it cannot have linear growth. Finally, we rule out exponential growth.
Since $G$ is quasi-transitive, there exists $R<\oo$ such that, for all edges $\langle x,y\rangle$ of $G$, 
the distance between $x$ and $y$ in $\RR^2$ is no greater than $R$. Therefore, the $n$-ball
centred at vertex $v$ is contained in $B_n(v) :=v+[-nR,nR]^2$. By quasi-transitivity again, there exists $A<\oo$ such that, for all $v$, $B_n(v)$ contains
no more than $A(nR)^2$ vertices.

The theorem of \cite{ST97} may now be applied to find that $\Aut(G)$ has a finite-index subgroup
$F$ isomorphic to $\ZZ^2$. Thus $F$ is generated by a pair of automorphisms which,
by Theorem \ref{p21}(c),  extend to 
isometries of the embedding of $G$.
\end{proof}

\begin{proof}[Proof of Lemma \ref{l73}]
By insertion-tolerance and ergodicity, the four random variables 
featuring in \eqref{eq:-9} are each $\mu$-a.s.\
constant  and take values in $\{0,1\}$. 
By Theorem \ref{thm:embedqt} and \cite[Thm 1.5]{DRT},
\begin{equation}\label{eq:-8}
\mu\bigl((N_{G_1},\ol N_{G_2}) = (1,1)\bigr)=0.
\end{equation}
Arguments related to  but weaker than \cite[Thm 1.5]{DRT} are found in \cite{BR07,GP,Men1,Men2,Sheff}. 
Note that
\cite[Thm 1.5]{DRT} deals with bond percolation on planar graphs, 
whereas \eqref{eq:-8}  is concerned with site percolation on non-planar graphs.
The site model may be handled either by adapting the arguments of \cite{DRT} to site models, or
by applying \cite[Thm 1.5]{DRT} to the one-dependent bond model
constructed in the manner described in Section \ref{sec:sbond} (see 
\eqref{eq:sitebond} and Proposition \ref{ll2}). 
Non-planarity is avoided by working with the facial graphs
of Section \ref{sec:sbond}. 
The remaining part of \eqref{eq:-9} follows
from the fact that $\ol N_{M_*}=1$ $\mu$-a.s.\ on the event $\{\ol N_M=1\}$.
\end{proof}

\begin{corollary}\label{l72}
Let $G\in\sQ$ be amenable and one-ended, and consider site percolation on $G$.
Then $\PP_{\frac12}(N=0) = 1$.
\end{corollary}

\begin{proof}
Suppose that $\PP_{\frac12}(N\ge 1)>0$, so that $\PP_{\frac12}(N\ge 1)=1$ by ergodicity.
By amenability and symmetry,  we have that $\PP_{\frac12}(N=\ol N=1)=1$. This contradicts Lemma \ref{l73}.
\end{proof}

\begin{lemma}\label{lm33}
Let $M=(V,E)\in\sQ$ be amenable, one-ended, and embedded canonically in $\sH$, and let $(G_1,G_2)$
be a matching pair derived from $M$. We have for site percolation that
$\PP_p(\ol N_{G_2}=1)=1$ for $p<\pcs (G_1)$.
\end{lemma}

\begin{proof}
Let $p\in (0,\pcs (G_1))$ be such that $\PP_{p}(\ol N_{G_2}=1)<1$. By 
amenability and ergodicity, we have that 
\begin{equation}\label{eq:new101}
\PP_{p}(\ol N_{G_2}=0)=1.
\end{equation}
Therefore, $\PP_p(N_{G_1}=\ol N_{G_2}=0)=1$.
There is a standard geometrical argument based on subcritical exponential decay 
that leads to a contradiction, as follows. 

Fix a vertex $v_0$ of $M=(V,E)$, and let $\ga$ be a semi-infinite geodesic of $M$ with endvertex $v_0$.
Let $n\ge 1$, and let $\La_n=\{u\in V: d_M(u,v_0)\le n\}$.
By \cite[Prop.\ 2.1]{K82}, 
if $\pd \La_n$ intersects no infinite closed path of $G_2$, 
there exists some open circuit of $G_1$ with $\La_n$ in its inside. 
There exists $c=c(M)>0$ such that, if the last event occurs, then
for some $k\ge 1$ and some $v\in\ga \cap \pd \La_{n+k}$, we have that $v$ lies in an
open path of $G_1$ of length at least $c(n+k)$.
By \cite[Thm 3]{AV07} for example, and \eqref{eq:new101},
there exist $A, a>0$ such that,  
$$
1 \le \sum_{k\ge 1} Ae^{-a(n+k)}.
$$
This cannot hold for large $n$, and the lemma is proved.
\end{proof}

We turn to equation \eqref{eq:-5}. In this amenable case, this is equivalent to the following extension of
classical results of Sykes and Essam \cite{SE64} (see also  \cite{vdB81}).

\begin{theorem}\label{t31-0}
Let $M=(V,E)\in\sQ$ be amenable, one-ended, and embedded canonically in $\sH$, and let $(G_1,G_2)$
be a matching pair derived from $M$. Then
$$
\pcs (G_1)+\pcs (G_2)=1.
$$
\end{theorem}

\begin{proof}
By Lemma \ref{l73}, whenever $p>\pcs (G_1)$, we have $1-p\leq \pcs (G_2)$, 
which implies $\pcs (G_1)+\pcs (G_2)\geq 1$.
By Lemma \ref{lm33}, whenever $p<\pcs (G_1)$, 
we have $1-p\geq \pcs (G_2)$, which implies $\pcs (G_1)+\pcs (G_2)\leq 1$. 
\end{proof}

\section{Non-amenable graphs with one end}
\label{ss5}

In this section, we prove Theorem \ref{m1}(a) for non-amenable, 
one-ended graphs $G=(V,E)\in\sQ$; 
see Remark \ref{rem:enh} for an explanation of part (b) of the theorem. 
Recall from Section \ref{ssec:percnot} 
that $N$ (\resp, $\ol N$)  denotes the number of $1$-clusters (\resp, $0$-clusters).

\begin{lemma}\label{l35}
Let $M=(V,E)\in\sQ$ be one-ended and embedded canonically in the hyperbolic plane, and let $(G_1,G_2)$
be a matching pair derived from $M$.  For $\om\in\Om_V$,
\begin{equation*}
\PP_p\bigl((N_{G_1},\ol N_{G_2})\in\{(0,1),(1,0),(\infty,\infty)\}\bigr)=1.
\end{equation*}
\end{lemma}

\begin{proof}
We fix a canonical embedding of $M$. By Proposition \ref{ll2},  
\begin{equation*}
N_{G_1}(\om) = N_{\what G_1}(\be_{\om_1}), \qq N_{G_2}(1-\om) = N_{\what G_1^+}(\be_{\om_1}^+).
\end{equation*}
By Lemma \ref{lem9}, the law of  $\be_{\omega_1}$ is weakly deletion-tolerant, 
weakly-insertion tolerant, and ergodic, and the claim follows by  Lemma \ref{ll25}.
\end{proof}

\begin{proof}[Proof of Theorem \ref{m1}(a)]
By Lemmas \ref{lsu} and \ref{l32}, we have the following for site percolation on either $G_1$ or $G_2$:
\begin{align*}
\text{if } p < \pc,\q &\PP_p(N = 0)=0,\\
\text{if } \pc<p<\pu,\q &\PP_p(N=\infty)=1,\\
\text{if } p > \pu,\q &\PP_p(N=1)=1,
\end{align*}
where $\pc$, $\pu$ are the critical values appropriate to the graph in question.

By Lemma \ref{l35},  $N_{G_1}=1$ if and only if $\ol N_{G_2}=0$, whence 
$\pu(G_1)=1-\pc(G_2)$. 
\end{proof}

\begin{corollary}\label{c56}
Let $G\in\sQ$ be one-ended and embedded canonically in $\sH$, and suppose $G$ is non-amenable. 
Then
\begin{equation*}
\PP_p\bigl((N,\ol N)\in \{(0,0),(0,1),(1,0),(0,\infty),(\infty,0),(\infty,\infty)\}\bigr)=1.
\end{equation*}
\end{corollary}

\begin{proof}
By Lemma \ref{l32}, $\PP_p$-a.s.\ the pair $(N,\ol N)$ takes some given value in the set $\{0,1,\oo\}^2$.
We need to eliminate the vectors $(1,1)$, $(1,\oo)$, and $(\oo,1)$.
If either of the vectors $(1,1)$ and $(1,\oo)$ have strictly positive probability, then
$\PP_p(N=1,\, \ol N_*\ge 1)>0$, in contradiction of Lemma \ref{l35}
applied to the matching pair $(G,G_*)$.
By symmetry,  $\PP_p((N,\ol N)\neq (\infty,1))=1$, and the corollary follows.
\end{proof}

\section{Proof of Theorem \ref{sc1}}\label{psc1}

Let $G$ be a graph satisfying the assumptions of the theorem. 
We work with the largest finite connected subgraph $G_B$ of $G$ contained in a large 
bounded region $B$  
(with boundary $\pd B$) of the natural geometry of $G$, and shall
let $B$ expand to fill the space. The numbers of finite faces, vertices, edges of $G_B$ satisfy Euler's formula:
$f_B+v_B=e_B+1$. Since the smallest possible face is a triangle, we have $f_B\le \frac23 e_B$; 
since the degree of interior vertices is $7$ or more, there exists $c>0$ such that $e_B \ge \frac72 (v_B-  c| \pd B|)$.
This contradicts Euler's formula unless  $e_B/|\pd B|$ is bounded above, 
which is to say that the natural geometry is the hyperbolic plane.
Hence, $G$ is non-amenable.  By \cite[Thm 2]{HP19}, we have $\pcs=\pcs(G)< \frac12$.

By the symmetry of the interval $(\pcs,1-\pcs)$ around $\frac12$, it suffices to show that $\PP_p(N=\oo)=1$
for $p\in (\pcs,1-\pcs)$. This in turn is implied by Lemma \ref{lsu} and the inequality 
\begin{equation}\label{simple}
1-\pcs\le \pus.
\end{equation}
Inequality \eqref{simple} holds by \eqref{eq:-6} when is $G$ non-amenable and one-ended.
In the remaining case when $G$ has infinitely many ends, \eqref{simple} is trivial since $\pus=1$ by Theorem \ref{m3}.

\section{Proof of Theorem \ref{sc2}}
\label{psc2}

Let $G$ be a graph satisfying the assumptions of
the theorem, and embedded canonically.
By Lemma \ref{l32}, symmetry, and the assumption $\PP_{\frac12}(N\ge 1)=1$, 
\begin{equation}\label{eq:new20}
\PP_{\frac12}\bigl((N,\ol N)\in\{(1,1),(\infty,\infty)\}\bigr)=1.
\end{equation}

By Theorem \ref{p13}, the following four cases may occur:

\begin{letlist}
\item $G$ is amenable and one-ended. By Lemma \ref{l73}, $\PP_{\frac12}(N=0)=1$. Hence, in this case, the hypothesis of the theorem is invalid. 

\item $G$ is non-amenable and one-ended. By Corollary \ref{c56} and \eqref{eq:new20}, 
subject to the percolation assumption, we have $\PP_{\frac12}(N= \ol N = \oo)=1$.

\item $G$ has two ends. By Theorem~\ref{m2}, $\pcs =1$. Hence $\PP_{\frac12}(N=0)=1$, and the hypothesis is invalid.

\item $G$ has infinitely many ends. By Theorem~\ref{m3}, $\pus =1$. Under the hypothesis of the theorem,
it follows by symmetry that $\PP_{\frac12}((N,\ol N)=(\infty,\infty))=1$. 
\end{letlist}

\section*{Acknowledgements} 
The authors thank the two referees for their helpful and detailed remarks.
ZL's research was supported by National Science Foundation grant 1608896 and Simons Collaboration Grant 638143.

\providecommand{\bysame}{\leavevmode\hbox to3em{\hrulefill}\thinspace}
\providecommand{\MR}{\relax\ifhmode\unskip\space\fi MR }
\providecommand{\MRhref}[2]{%
  \href{http://www.ams.org/mathscinet-getitem?mr=#1}{#2}
}
\providecommand{\href}[2]{#2}

\end{document}